\documentclass[11pt,a4paper]{amsart}

\usepackage{lmodern}
\usepackage{microtype}
\usepackage[english]{babel}
\usepackage{hyperref}
\usepackage[msc-links]{amsrefs}

\theoremstyle{plain} 
\newtheorem{theorem}{Theorem}
\newtheorem{lemma}[theorem]{Lemma}

\DeclareMathOperator{\mre}{Re}

\begin{document} 
\title{Harald Bohr's splitting theorem}
\date{\today}

\author{Viktor Andersson} 
\address{Department of Mathematical Sciences, Norwegian University of Science and Technology (NTNU), 7491 Trondheim, Norway} 
\email{viktor.andersson@ntnu.no}

\author{Ole Fredrik Brevig} 
\address{Department of Mathematical Sciences, Norwegian University of Science and Technology (NTNU), 7491 Trondheim, Norway} 
\email{ole.brevig@ntnu.no}

\author{Athanasios Kouroupis} 
\address{Department of Mathematics, KU Leuven, Celestijnenlaan 200B, 3001, Leuven, Belgium} 
\email{athanasios.kouroupis@kuleuven.be}

\begin{abstract}
    We present a new elementary proof of a theorem due to Harald Bohr,
    which states that an unbounded, analytic, and almost periodic function in a half-plane can be written as the sum of two analytic functions: the first is unbounded and periodic, while the second is bounded and almost periodic. The proof is based on a well-known arithmetical property of translation numbers of almost periodic functions.
\end{abstract}


\thanks{Andersson and Brevig are supported by Grant 354537 of the Research Council of Norway. Kouroupis is supported by Grant 1203126N of the Research Foundation -- Flanders (FWO)}

\maketitle

\section{Introduction}
The purpose of this note is to draw attention to a result in the theory of analytic almost periodic functions that appears to have been lost in the mists of time. Almost periodic functions were introduced by Harald Bohr in a series of papers \cites{Bohr1925A,Bohr1925B,Bohr1926} that appeared in 1925 and 1926. Bohr continued to study almost periodic functions for the rest of his life, and it is one of his later contributions \cite{Bohr1943} that we are now concerned with.

This tract came to be during a dramatic time in Bohr's life. It was submitted to the Royal Danish Academy of Sciences and Letters on 7 July 1943 and appeared in print on 25 November 1943. In the mean time, Bohr fled Denmark on 29 September (see e.g.~\cite{Pais1991}*{pp.~487--488}).

Let us set the stage by going over the central definitions of almost periodic functions. We consider functions $f$ defined on half-planes
\[\mathbb{C}_\kappa = \{s=\sigma+it\,:\,\sigma>\kappa\}\]
and their vertical translations $V_\tau f(s) = f(s+i\tau)$ for real numbers $\tau$. We say that $\tau$ is a \emph{translation number of} $f$ \emph{belonging to} $\varepsilon > 0$ if the estimate $\lvert V_\tau f-f\rvert \leq \varepsilon$ holds in $\mathbb{C}_\kappa$, and we let $E_f(\varepsilon)$ denote the collection of such $\tau$. The union of $E_f(\varepsilon)$ over all $\varepsilon>0$ is called the \emph{translation module} of $f$ and is denoted $T_f$. Recall that a set of real numbers is called \emph{relatively dense} if there is a real number $L>0$ such that the set has non-empty intersection with any interval of length $L$. A continuous function $f$ on $\mathbb{C}_\kappa$ is called \emph{almost periodic} if $E_f(\varepsilon)$ is relatively dense for every $\varepsilon>0$.

We are interested in analytic almost periodic functions that are unbounded as $\mre{s} \to \infty$. It may be instructive to consider a simple example. If $\lambda$ is a real number, then the analytic function $f_\lambda(s) = \exp(\lambda s)$ is \emph{periodic} and, consequently, almost periodic in any half-plane $\mathbb{C}_\kappa$. Note that if $\lambda>0$, then $f_\lambda$ is unbounded as $\mre{s}\to\infty$, so that
\[T_{f_\lambda} = \left\{2\pi k/\lambda\,:\,k\in\mathbb{Z}\right\}.\]
If $\lambda_1 \geq \lambda_2 > 0$ and $f = f_{\lambda_1} + f_{\lambda_2}$, then either $\lambda_1/\lambda_2$ is rational (and $f$ is periodic) or $T_f=\{0\}$.

In our statement of Bohr's theorem, we have chosen to fix a reference half-plane $\mathbb{C}_0$ and work in every smaller half-plane. Bohr's original (and equivalent) formulation is slightly different.

\begin{theorem}[Bohr's splitting theorem] \label{thm:bohrsplit} The following are equivalent:
    \begin{enumerate}
        \item[\normalfont (i)] $f$ is unbounded, analytic, and almost periodic in $\mathbb{C}_\kappa$ for every $\kappa>0$. 
        \item[\normalfont (ii)] $f=p+b$, where 
        \begin{itemize}
            \item $p$ is unbounded, analytic, and periodic in $\mathbb{C}_\kappa$ for every $\kappa>0$,
            \item $b$ is bounded, analytic, and almost periodic in $\mathbb{C}_\kappa$ for every $\kappa>0$.
        \end{itemize}
    \end{enumerate}
\end{theorem}

The main takeaway of Bohr's splitting theorem is that the unbounded part of $f$ is contained in the periodic function $p$. The main benefit of this is that $p$ is rather easy to analyze as it may be expressed in terms of the Laurent series of an analytic function in the punctured unit disc, namely as
\begin{equation} \label{eq:laurent}
    p(s) = \sum_{n \in \mathbb{Z}} a_n \exp(-\lambda n s)
\end{equation}
for some $\lambda>0$. Note that \eqref{eq:laurent} converges absolutely at every point in $\mathbb{C}_0$. Since $p$ is unbounded in $\mathbb{C}_\kappa$ for every $\kappa>0$, there must be some $n<0$ such that $a_n \neq 0$. We refer to the part of the Laurent series \eqref{eq:laurent} comprised of the terms with $n<0$ as the \emph{unbounded part} of $p$. The next result addresses in what sense the splitting $f=p+b$ is unique. We say that the \emph{period} of a non-constant periodic function $p$ is the smallest $\tau>0$ such that $V_\tau p = p$.

\begin{theorem}[Uniqueness in Bohr's splitting theorem] \label{thm:bohrunique}
    Let $f$ be unbounded, analytic, and almost periodic in $\mathbb{C}_\kappa$ for every $\kappa>0$.
    \begin{enumerate}
        \item[\normalfont (a)] If $f=p+b$ and if $t>0$ is the period of the unbounded part of $p$, then 
\[T_f = \{tk\,:\,k \in \mathbb{Z}\}.\]
        \item[\normalfont (b)] If $f = p_1 + b_1$ and $f = p_2 + b_2$, then the unbounded parts of $p_1$ and $p_2$ coincide. 
    \end{enumerate}
\end{theorem}

Note that the equality in part (a) is independent of which half-plane $\mathbb{C}_\kappa$ with $\kappa>0$ that we take the translation module with respect to. Theorem~\ref{thm:bohrunique} will be established by a fairly simple analysis of the Laurent series \eqref{eq:laurent}.

The main contribution of the present note is a new proof of Bohr's splitting theorem that does not rely on the Fourier-analytic machinery developed in the initial papers \cites{Bohr1925A,Bohr1925B,Bohr1926}. We were inspired by a paper of Bochner~\cite{Bochner1963}, which contains a simple proof of one of the fundamental results from \cite{Bohr1926}. Our proof is short and elementary: it relies only on a few basic facts from complex analysis. It moreover identifies an
arithmetical property of translation numbers as the heart of the matter.

\begin{lemma} \label{lem:aritprog}
    If $f$ is almost periodic and uniformly continuous in $\mathbb{C}_\kappa$, then
    \[\{tk\,:\,k\in\mathbb Z\} \cap E_f(\varepsilon)\]
    is relatively dense for all $t>0$ and every $\varepsilon>0$.
\end{lemma}

Lemma~\ref{lem:aritprog} is well-known and (a version of it) can be found in \S I.11 of Besicovitch \cite{Besicovitch1955}, as well as in \S 9 of \cite{Bohr1925A}. In order to make the present note completely self-contained, we include the proof below. The key point is that the uniform continuity of $f$ ensures that for every $\varepsilon>0$ there is a $\delta>0$ such that the interval $[-\delta,\delta]$ is contained in $E_f(\varepsilon)$.

The implication (ii) $\implies$ (i) of Theorem~\ref{thm:bohrsplit} is a fairly direct consequence of Lemma~\ref{lem:aritprog} when applied to $b$. We cannot use Lemma~\ref{lem:aritprog} for the reverse implication, since $f$ is not uniformly continuous in $\mathbb{C}_\kappa$. However, $f$ is bounded and uniformly continuous in any vertical strip
\[\mathbb{S}_{\alpha,\beta} = \{s = \sigma+it\,:\, \alpha<\sigma<\beta\}\]
with $\kappa<\alpha<\beta$ due to almost periodicity (see Lemma~\ref{lem:strip} below). This yields the following local version of Lemma~\ref{lem:aritprog}, where we write $E_f(\varepsilon,\mathbb{S}_{\alpha,\beta})$ for the translation numbers of $f$ belonging to $\varepsilon$ with respect to $\mathbb{S}_{\alpha,\beta}$.

\begin{lemma} \label{lem:aritprogstrip}
    If $f$ is almost periodic in $\mathbb{C}_\kappa$ and if $\kappa<\alpha<\beta$, then
    \[\{tk\,:\,k\in\mathbb Z\} \cap E_f(\varepsilon,\mathbb{S}_{\alpha,\beta})\]
    is relatively dense for all $t>0$ and every $\varepsilon>0$.
\end{lemma}

We will utilize Lemma~\ref{lem:aritprogstrip} and the fact that $f$ is bounded in $\mathbb{S}_{\alpha,\beta}$ to prove the existence of the \emph{periodic components} $P_\tau f$ of an almost periodic function.

\begin{theorem} \label{thm:percomp}
    If $f$ is an almost periodic function in $\mathbb{C}_\kappa$ and $\tau>0$, then
    \begin{equation} \label{eq:percomp}
        P_\tau f(s) = \lim_{k\to\infty} \frac{1}{k} \sum_{j=0}^{k-1} V_{j\tau} f(s)
    \end{equation}
    converges uniformly in $\mathbb{S}_{\alpha,\beta}$ for every $\kappa<\alpha<\beta$. It holds that 
    \[V_\tau P_\tau f(s) = P_\tau f(s),\] 
    so that $f$ is periodic in $\mathbb{C}_\kappa$ with period at most $\tau$.
\end{theorem}

The idea in our proof of the implication (i) $\implies$ (ii) in Theorem~\ref{thm:bohrsplit} is to extract the periodic function $p$ as a periodic component $P_\tau f$ using Theorem~\ref{thm:percomp}. The question is then which $\tau>0$ to choose in order to ensure that the unbounded part of $f$ is contained in $P_\tau f$. A hint can be found in Theorem~\ref{thm:bohrunique}~(a): we can choose any $\tau>0$ from the translation module $T_f$. The justification comes from the assumption that $f$ is analytic via a suitable application of the maximum principle.

\subsection*{Organization} This note is comprised of two additional sections. Section~\ref{sec:proofs} contains the proofs of the results stated above, while Section~\ref{sec:bohr} is devoted to a short discussion of Bohr's proof of Theorem~\ref{thm:bohrsplit}.

\section{Proofs} \label{sec:proofs}
We begin with the well-known Lemma~\ref{lem:aritprog}. Our argument is a streamlined version of the proof of Lemma~10 from \S I.1 in Besicovitch~\cite{Besicovitch1955}. We will use the following basic consequence of the triangle inequality: If $\tau_1$ and $\tau_2$ are translation numbers belonging to $\varepsilon_1$ and $\varepsilon_2$, then $\tau_1 \pm \tau_2$ are translation numbers belonging to $\varepsilon_1+\varepsilon_2$.

\begin{proof}[Proof of Lemma~\ref{lem:aritprog}]
    Fix $t>0$ and $\varepsilon>0$. Since $f$ is almost periodic, there is a number $L>0$ such that every interval of length $L$ contains an element of $E_f(\varepsilon/4)$. For each integer $k$, we can therefore find $\tau_k$ in $E_f(\varepsilon/4)$ such that 
    \[\lvert t k - \tau_k\rvert \leq L/2.\] Using next the uniform continuity of $f$, we infer that there is $\delta>0$ such that $[-\delta,\delta]$ belongs to $E_f(\varepsilon/2)$ and such that $L/\delta=d$ is an integer. Divide the interval $[-L/2,L/2]$ into $d$ disjoint intervals of length $\delta$, and define the equivalence relation $j \sim k$ on $\mathbb{Z}$ by the requirement that $tj - \tau_j$ and $tk - \tau_k$ belong to the same interval of length $\delta$. If $j \sim k$, then 
    \[\lvert t(j-k) - (\tau_j-\tau_k) \rvert \leq \delta,\]
    so that $t(j-k)$ belongs to $E_f(\varepsilon)$ since $\tau_j-\tau_k$ belongs to $E_f(\varepsilon/2)$. Let $k'$ denote the integer of smallest modulus in the equivalence class $[k]$ and define
    \[A = \{t (k-k') \,:\, k \in \mathbb{Z}\}.\]
    It is clear that $A$ is contained in the intersection of $\{t k \,:\, k \in \mathbb{Z}\}$ and $E_f(\varepsilon)$. Let $K$ denote the maximum for $|k'|$ for the at most $d$ equivalence classes. Since every interval of length at least $t(2K+1)$ has nonempty intersection with $A$, it follows that $A$ is relatively dense.
\end{proof}


We can now use Lemma~\ref{lem:aritprog} to establish the easy implication of Bohr's splitting theorem.

\begin{proof}[Proof of the implication {\normalfont(ii)}$\implies${\normalfont(i)} in Theorem~\ref{thm:bohrsplit}]
    Since $b$ is bounded and analytic in $\mathbb{C}_\kappa$ for every $\kappa>0$, it is uniformly continuous in $\mathbb{C}_\kappa$ for every $\kappa>0$ by Cauchy's estimate. Fix $\kappa>0$ and let $t>0$ be the period of $p$. It is plain that $E_f(\varepsilon)$ contains the intersection of $\{tk\,:\,k\in\mathbb Z\}$ and $E_b(\varepsilon)$, which is relatively dense for every $\varepsilon>0$ due to Lemma~\ref{lem:aritprog}.
\end{proof} 

We need the following standard (see e.g. \S I.1 in Besicovitch \cite{Besicovitch1955}) result alluded to above. Its proof amounts to noticing that the continuity of $f$, baked into the assumption of almost periodicity, can be extended to boundedness and uniform continuity in any compact subset of $\mathbb{C}_\kappa$, which can in turn be extended to a vertical strip using almost periodicity. 

\begin{lemma} \label{lem:strip}
    If $f$ is almost periodic in $\mathbb{C}_\kappa$, then $f$ is bounded and uniformly continuous in $\mathbb{S}_{\alpha,\beta}$ for every $\kappa<\alpha<\beta$.
\end{lemma}

In view of the uniform continuity asserted in Lemma~\ref{lem:strip}, the proof of Lemma~\ref{lem:aritprog} also applies to Lemma~\ref{lem:aritprogstrip}.

\begin{proof}[Proof of Theorem~\ref{thm:percomp}]
    Fix $\kappa<\alpha<\beta$. We will establish the stronger claim that there for every $\varepsilon>0$ is a positive integer $K$ such that
\begin{equation}\label{eq:periodic-component-cauchy}
        \left\lvert\frac{1}{\lvert J_1\rvert}\sum_{j\in J_1}V_{j\tau}f-\frac{1}{\lvert J_2\rvert}\sum_{j\in J_2}V_{j\tau}f\right\rvert\leq\varepsilon
    \end{equation}
    in $\mathbb S_{\alpha,\beta}$ for all sets $J_1$ and $J_2$ of at least $K$ consecutive integers. 

    By Lemma~\ref{lem:aritprogstrip}, there is an integer $d$ such that among any $d$ consecutive integers, there is at one least integer $m$ such that $m\tau$ belongs to $E_f(\varepsilon/6,\mathbb{S}_{\alpha,\beta})$. Let $M$ be the least upper bound of the modulus of $f$ in $\mathbb{S}_{\alpha,\beta}$ and recall that $M<\infty$ due to Lemma~\ref{lem:strip}. 
    
    Let $D \geq d$ be an integer to be chosen later. Let $J_1$ and $J_2$ be sets consisting of $D$ consecutive integers and let $a_1$ and $a_2$ denote the smallest integer in each set. There are integers $a_1 \leq m_1 < a_1+d$ and $a_2 \leq m_2 < a_2 + d$ such that both $m_1 \tau$ and $m_2 \tau$ belong to $E_f(\varepsilon/6,\mathbb{S}_{\alpha,\beta})$. We next estimate
    \begin{multline*}
        \left\lvert\sum_{j\in J_1}V_{j\tau}f-\sum_{j\in J_2}V_{j\tau}f\right\rvert \\ \leq\left\lvert\sum_{j\in J_1}(V_{j\tau}f-V_{(j+m_2-m_1)\tau}f)\right\rvert+\left\lvert\sum_{j\in J_1+m_2-m_1}V_{j\tau}f-\sum_{j\in J_2}V_{j\tau}f\right\rvert.
    \end{multline*}
    The first term can be bounded by $(\varepsilon/3) D$, since $(m_2-m_1)\tau$ belongs to $E_f(\varepsilon/3,\mathbb{S}_{\alpha,\beta})$. The second term can be bounded by $2d M$, since $D \geq d$. We now choose an integer $D$ satisfying $D \geq d\max(1,6M/\varepsilon)$ to obtain
    \begin{equation} \label{eq:compest1}
        \left\lvert\frac{1}{\lvert J_1\rvert}\sum_{j\in J_1}V_{j\tau}f-\frac{1}{\lvert J_2\rvert}\sum_{j\in J_2}V_{j\tau}f\right\rvert \leq \frac{\varepsilon}{3}+\frac{2dM}{D} \leq \frac{2}{3}\varepsilon,
    \end{equation}
    since $|J_1|=|J_2|=D$. 
    
    Suppose next that there are positive integers $n_1$ and $n_2$ such that $J_1$ and $J_2$ consist of $n_1 D$ and $n_2 D$ consecutive integers, respectively. Notice that the estimate \eqref{eq:compest1} also holds for these sets, since it can be obtained by applying the case $n_1=n_2=1$ of the estimate $n_1 n_2$ times. 
    
    This yields the desired estimate \eqref{eq:periodic-component-cauchy}, but only in the case $J_1$ and $J_2$ are sets of $n_1D$ and $n_2 D$ consecutive integers. To obtain the general case, suppose that $J$ is a set of consecutive integers such that $n D \leq \lvert J\rvert < (n+1)D$, and let $J_0$ denote the $nD$ first integers in $J$. Straightforward estimates yields that
    \begin{equation} \label{eq:compest2}
        \left\lvert \frac{1}{\lvert J \rvert} \sum_{j \in J} V_{j\tau} f- \frac{1}{\lvert J_0\rvert} \sum_{j \in J_0} V_{j\tau} f \right\rvert \leq 2M\left(1-\frac{\lvert J_0 \rvert}{\lvert J \rvert}\right) \leq \frac{2M}{n} \leq \frac{\varepsilon}{6},
    \end{equation}
    if $n\geq N$ and for an integer $N \geq M/(12\varepsilon)$. If we set $K = ND$, then we obtain the desired estimate \eqref{eq:periodic-component-cauchy} after using \eqref{eq:compest2} twice and \eqref{eq:compest1} once.

    All that remains is to prove the final assertion. Since $f$ is bounded in $\mathbb{S}_{\alpha,\beta}$ for $\kappa<\alpha<\beta<\infty$ by Lemma~\ref{lem:strip}, it follows from the definition of $P_\tau f$ in \eqref{eq:percomp} and its convergence that $V_\tau P_\tau f(s) = P_\tau f(s)$ holds for all $s$ in $\mathbb{C}_\kappa$.
\end{proof}

One final ingredient is needed for the proof of Theorem~\ref{thm:bohrsplit}, namely the maximum principle. We let $M_\infty(\sigma,f)$ denote the supremum of $\lvert f(\sigma+it)\rvert$ over real numbers $t$. Note that Lemma~\ref{lem:strip} ensures that if $f$ is almost periodic in $\mathbb{C}_\kappa$, then $M_\infty(\sigma,f) < \infty$ for every $\kappa<\sigma<\infty$.

In the next result, we let $T_f(\mathbb{C}_\kappa)$ stand for the translation module of $f$ with respect to the half-plane $\mathbb{C}_\kappa$.

\begin{lemma} \label{lem:mp}
    Suppose that $f$ is analytic in $\mathbb{C}_0$. If $\tau$ belongs to the translation module $T_f(\mathbb{C}_\kappa)$ for some $\kappa>0$, then 
    \[\lvert V_\tau f(s) - f(s)\rvert \leq 2M_\infty(\kappa,f)\]
    for every $s$ in $\mathbb{C}_\kappa$.
\end{lemma}

\begin{proof}
    The assumption that $\tau$ belongs to the translation module $T_f(\mathbb{C}_\kappa)$ means that $V_\tau f - f$ is a bounded analytic function in $\mathbb{C}_\kappa$. By the maximum principle, we get that
    \[\lvert V_\tau f(s) - f(s)\rvert \leq M_\infty(\kappa,V_\tau f-f)\]
    for every $s$ in $\mathbb{C}_\kappa$. The stated estimate follows from the triangle inequality.
\end{proof}

We are now in a position to finish the proof of Bohr's splitting theorem.

\begin{proof}[Proof of the implication {\normalfont(i)}$\implies${\normalfont(ii)} in Theorem~\ref{thm:bohrsplit}]
    Let us start with a basic observation. Since $f$ is bounded in $\mathbb{S}_{\alpha,\beta}$ for every $0<\alpha<\beta$ due to Lemma~\ref{lem:strip} and since every $\tau$ is a translation number of a bounded function, it follows that the translation module $T_f$ does not depend on which half-plane $\mathbb{C}_\kappa$ with $\kappa>0$ we consider $f$ as a function on.

    Let therefore $\tau>0$ be a number in the translation module $T_f$ and let $p = P_\tau f$. The final assertion in Theorem~\ref{thm:percomp} implies that $p$ is periodic in $\mathbb{C}_0$, while the uniform convergence of \eqref{eq:percomp} ensures that $p$ is analytic in $\mathbb{C}_0$. This also shows that $b = f-p$ is analytic in $\mathbb{C}_0$. The task at hand is therefore to show that $b$ is almost periodic and bounded in $\mathbb{C}_\kappa$ for every $\kappa>0$. Let therefore $\kappa>0$ be fixed.

    Fix $\varepsilon>0$. It follows from the definition of $P_\tau f$ in \eqref{eq:percomp} that $E_f(\varepsilon/2,\mathbb{C}_\kappa)$ is a subset of $E_p(\varepsilon/2,\mathbb{C}_\kappa)$, which when combined with the triangle inequality yields that $E_b(\varepsilon,\mathbb{C}_\kappa)$ contains $E_f(\varepsilon/2,\mathbb{C}_\kappa)$. Hence $b$ is almost periodic since $f$ is almost periodic.

    We have not yet used that $\tau>0$ is a translation number of $f$. If we bring this assumption into play and use Lemma~\ref{lem:mp}, then we find that
    \[\lvert b(s) \rvert \leq \liminf_{k\to\infty} \frac{1}{k} \sum_{j=0}^{k-1} \left\lvert f(s)-V_{j\tau} f(s)\right\rvert \leq 2M_\infty(\kappa,f)\]
    for every $s$ in $\mathbb{C}_\kappa$.  The the right-hand side is finite due to Lemma~\ref{lem:strip}.
\end{proof}

We now turn to the proof of the uniqueness theorem, which involves an analysis of the Laurent series \eqref{eq:laurent} of the periodic component.

\begin{proof}[Proof of Theorem~\ref{thm:bohrunique}]
   We begin with (a) and fix $\kappa>0$. By the triangle inequality, the addition of a bounded function does affect the translation module. If we write $f=p+b$ and $p=p_u+p_b$ for the unbounded and bounded parts of $p$, then
   \[T_f(\mathbb{C}_\kappa) = T_{p_u}(\mathbb{C}_\kappa).\]
   Let $t>0$ be the period of $p_u$. Inspecting the Laurent series \eqref{eq:laurent}, it is clear that $E_{p_u}(\varepsilon,\mathbb{C}_\kappa) = \{tk\,:\,k \in \mathbb{Z}\}$ for any $\varepsilon>0$. The proof of (a) is complete.

   Now we turn to (b) and we again fix $\kappa>0$. If $f=p_1+b_1$ and $f=p_2+b_2$, then the difference of the unbounded parts of $p_1$ and $p_2$ is bounded in $\mathbb{C}_\kappa$. Since they have the same period by (a), it follows that the difference of their unbounded parts must vanish identically in $\mathbb{C}_\kappa$.
\end{proof}

\section{Bohr's proof of the splitting theorem} \label{sec:bohr}
Bohr's proof is based on the deep result (see e.g. \S III.3 in Besicovitch \cite{Besicovitch1955}) that if $f$ is analytic and almost periodic in $\mathbb{C}_\kappa$ for every $\kappa>0$, then there is a countable set of real numbers $\Lambda$ such that $f$ has an associated Dirichlet series that uniquely determines it, 
\begin{equation} \label{eq:diriseri}
    f(s) \sim \sum_{\lambda \in \Lambda} a_\lambda \exp(-\lambda s).
\end{equation}
The Dirichlet series \eqref{eq:diriseri} should be compared to the Laurent series \eqref{eq:laurent}. Note that while the latter is absolutely convergent in $\mathbb{C}_0$, the former does not in general converge at any point. However, it is possible to use \eqref{eq:diriseri} to construct Dirichlet polynomials (i.e. finite series) that converge uniformly to $f$ in $\mathbb{S}_{\alpha,\beta}$ for every $0<\alpha<\beta$.

The approximation by Dirichlet polynomials can be used to give a short proof of Theorem~\ref{thm:percomp}, since it trivially holds for Dirichlet polynomials. Another consequence is that $f$ is bounded in $\mathbb{C}_\kappa$ for every $\kappa>0$ if and only if $\Lambda$ does not contain any negative numbers. 

Bohr uses the Dirichlet series \eqref{eq:diriseri} in the proof of both implications in Theorem~\ref{thm:bohrsplit}, but for the ``not quite trivial'' implication (ii)$\implies$(i) it is easily removed. Bohr's proof of this implication is similar to ours, although strictly speaking he uses the arithmetical property via  Lemma~\ref{lem:aritprogstrip} and not Lemma~\ref{lem:aritprog}.

Bohr's proof of the implication (i)$\implies$(ii) in Theorem~\ref{thm:bohrsplit} uses the Dirichlet series \eqref{eq:diriseri} in a much stronger sense. He begins by investigating the unbounded part of the Dirichlet series (the part corresponding to $\lambda<0$) and obtaining the formula for the translation module $T_f$ from   Theorem~\ref{thm:bohrunique}~(a). He then chooses the smallest positive number $t$ from $T_f$ and invokes Theorem~\ref{thm:percomp} to extract $p=P_t f$. He finally deduces that $b = f - p$ is bounded, since its Dirichlet series contains no terms with $\lambda<0$.

It is here our proof is significantly different. We give a self-contained proof of Theorem~\ref{thm:percomp} and use it to extract the unbounded part of $f$ without any information from  the Dirichlet series \eqref{eq:diriseri}. The key idea is Lemma~\ref{lem:mp}, which allows us to take $p = P_\tau f$ for any $\tau>0$ in the translation module $T_f$.

The period of $P_\tau f$ is $\tau/n$ for a positive integer $n$. Theorem~\ref{thm:bohrunique} asserts that the period of the unbounded part of $P_\tau f$ is always $t$. Hence choosing $\tau = kt$ for some $k>1$ does not affect the unbounded part of $P_\tau f$, but may lead to more terms in the bounded part of $P_\tau f$. Thus the choice $\tau = t$ made by Bohr is canonical in the sense that it makes the bounded part of $P_\tau f$ contain the fewest possible terms.

\bibliography{bst}

\end{document}